\documentclass[12pt,reqno]{amsart}
\usepackage[latin1]{inputenc}
\usepackage{amsmath,amssymb,amsfonts,amsthm,amscd,amstext,amsxtra,amsopn,array,url,verbatim,mathrsfs}
\usepackage{graphicx}
\usepackage{amsmath,amssymb,amsfonts,amsthm,amssymb,amscd,url,amstext,amsxtra,amsopn}
\usepackage{verbatim}
\usepackage{fullpage}
\usepackage{times}
\usepackage{ upgreek }
\usepackage{dsfont}
\usepackage{amsmath}
\usepackage{amsfonts}
\usepackage{amssymb}
\usepackage{amsmath, amsthm}     
\usepackage{amssymb}
\usepackage{subfigure}
\usepackage{url}
\usepackage{graphicx}
\usepackage{amsmath,amsfonts}
\usepackage{epsf,verbatim}
\usepackage[latin1]{inputenc}
\usepackage[T1]{fontenc}
\usepackage{lmodern}

\newtheorem{lemma}{Lemma}[section]

\newtheorem{theorem}{Theorem}[section]

\newtheorem*{conjecture}{Conjecture}

\theoremstyle{remark}
\newtheorem{remark}{Remark}[section]
\numberwithin{equation}{section}
\numberwithin{table}{section}
\numberwithin{figure}{section}

\begin{document}

\pagenumbering{arabic}


\title{{\textbf {On binary and quadratic divisor problems}}}
\author{Farzad Aryan}

\begin{abstract}
We study the shifted convolution sum of the divisor function and some other arithmetic functions.
\end{abstract}
\noindent
\maketitle
\section*{\textbf{Introduction}}
In this paper we are concerned with shifted convolution sums of several arithmetic functions. We divide the introduction into three parts. In the first part we discuss the binary divisor problem which plays an important role in bounding the forth moment of the zeta function on the critical line. In the second part we discuss
 the quadratic divisor problem which has applications to bounding more general $L$-functions. In the last part of the introduction we discus the application of the quadratic divisor problem to estimating the general shifted divisor problem, and the Lindel\"{o}f hypothesis.
\subsection{Binary convolution sums}
The binary additive divisor problem is related to calculation of
\begin{equation}
\label{0.1}
\sum_{n-m=h}d(n)d(m)f(n,m),\end{equation}
 where $f$ is a smooth function on $\mathbb{R}^{+}\times \mathbb{R}^{+}$ which oscillates mildly.
Vinogradov \cite{A.V} and Conrey and Gonek in \cite{C.G} conjectured that $$\sum_{n \leq X}d(n)d(n+h)= \textit{Main term} + O(X^{1/2 + \epsilon}),$$ uniformly for $h \leq X^{1/2},$ where the main term is of the form $XP(\log X),$ where $P$ is a quadratic polynomial whose coefficients are functions of $h$  . This problem begins with Ingham, who found an asymptotic with error term $o(X)$. Estermann \cite{Es} improved the error term to $O(X^{11/12 +\epsilon}).$ Using Weil's optimal bound on Kloosterman sums, Heath-Brown \cite{He} improved the error term to $O(X^{5/6 +\epsilon}).$ The final improvement on the error term with respect to $X,$ was obtained by Deshouillers and Iwaniec \cite{Des-Iwa}. For fixed $h$ they proved
 \begin{equation}
\label{d(n)}
\sum_{n \leq X}d(n)d(n+h)= \textit{Main term} + O(X^{2/3 + \epsilon}).
 \end{equation}
Further improvement in the $h$-aspect was obtained by Motohashi \cite{Mo}, where he proved a uniform result for $h \leq X^{64/39}.$ Finally, Meurman \cite{Meu} improved the range  to $h \leq X^{2-\epsilon}.$ This is the best result in the literature.\\
\\
\noindent In this article shifted convolution sum of the shape
\begin{equation}
 \mathcal{D}_{f, \lambda, \gamma} (a, b, h)= \sum_{an-bm=h}\lambda(n)\gamma(m)f(an, bm)
\end{equation}
shall be considered for sequences including the divisor function, Fourier coefficients of a primitive cusp form and the number of representations of an integer $n,$ as a sum of two squares. Let $a(n)$ be $n$-th Fourier coefficient of a primitive cusp form of weight $k$ and the level $N$
\begin{equation}
\label{cusp}
f(z)=\sum_{n=1}^{\infty} a(n)n^{(k-1)/2}e(nz).
\end{equation}
The shifted convolution sum for $a(n)$ is
\begin{equation}
\label{a(n)}
\mathcal{D}_{f, a, a}(\alpha, \beta, h):=\sum_{\alpha n-\beta m=h}a(n)a(m)f(\alpha n, \beta m).
\end{equation}
Blomer \cite{Bl} proved that if $f$ is supported on $[X,2X] \times [X,2X]$ and has decaying partial derivatives, satisfying
 \begin{equation}
\label{0.1}
\frac{\partial^{i+j}}{\partial x^{i} \partial y^{j}} f(x, y) \ll \frac{1}{X^{i+j}},
\end{equation}
then $$\mathcal{D}_{f, a, a}(\alpha, \beta, h) \ll X^{1/2 +\theta +\epsilon}.$$
The Ramanujan-Petersson conjecture predicts $\theta =0$ and the Weil bound for Kloosterman
sums gives $\theta \leq 1/4.$ Kim and Shahidi \cite{Kim-Shah} proved $\theta \leq 1/9,$  and the current best bound is $\theta \leq 7/64,$ due to Kim and Sarnak \cite{Kim-Sar}. \\
\\
\noindent For the sum of two squares we have $$r(n)=\# \{(x,y) : x^2+y^2=n\} = 4\sum_{d|n} \chi_4(d),$$
where $\chi_4$ is the non principal character modulo $4.$ For odd $h,$ Iwaniec \cite{Iw-r(n)}, by employing spectral theory, proved  $$\sum_{n\leq X} r(n)r(n+h)= 8\big(\sum_{d|h}\frac{1}{d}\big)X + O(h^{1/3}X^{2/3}),$$ and Chamizo \cite{Chaz} gave a conditional result for general $h.$ \\
\\
 \noindent There is a major difference between sequences like $d(n)$ or $r(n)$ and $a(n)$. We will explain it as follows. For
$a(n)$ we have  $$\sum_{n \leq X} a(n)e(n\alpha) \ll \sqrt X \log X,$$ while the same sum  obtained by replacing $a(n)$ with $d(n)$ or $r(n)$, depends on $\alpha$, has  main terms bigger than $\sqrt{X}$.  This difference makes it harder to deal with shifted convolution sums of sequences like $d(n)$ or $r(n)$. More precisely, the circle method developed by Jutila \cite{J} is very powerful to calculate the shifted convolution sums of coefficients of modular or Mass forms of $SL(2, \mathbb{Z})$ and even $SL(3, \mathbb{Z})$,( see \cite{Munshi}).  However because of the difference mentioned, the Jutila circle method is not useful for shifted convolution sums of the sequences with the main terms. The purpose of this part of this article is to develop the $\delta$-method of  Duke and Friedlander and Iwaniec \cite{DFI} in order to handle shifted convolution sums of these sequences, with good error terms. The key ingredient is using the Voronoi type summation formula to bring up a Kloosterman sums inside the circle method. Then instead of using Weil bound on Kloosterman sums we will get a better error term by means of the Kuznetsov trace formula \cite{Des-Iwa-Kuz} Theorem 1.  \\
\\
\noindent
In this direction we prove
\begin{theorem}
\label{Thm-1}
Let $f$ be a smooth function supported in $[X, 2X]\times[X, 2X]$ satisfying \eqref{0.1}. then for $\epsilon>0$ and $h \ll X^{1-\epsilon},$ we have
\begin{equation}
\label{Thm-1-eq}
\sum_{n-m=h}d(n)d(m)f(n, m)= \text{Main term} + O(X^{1/2 + \epsilon}h^{\theta}).
\end{equation}
Where the Ramanujan Petersson conjecture predict that $\theta=0$ and the Main term is the same as \cite[Equation 5]{DFI} with $a, b=1$ .
\end{theorem}
\noindent This improves on Meurman's result \cite{Meu}, $O(X^{1/2 +\epsilon}h^{1/8 + \theta/2})$, for the weight function $f$ satisfying \eqref{0.1}. (See page 238 of \cite{Meu} with $N \asymp X.$).\\
\\
\noindent Another example of sequences with main terms is obtained using a Dirichlet character. Let $\tau_{\chi}(n)= \sum_{d|n} \chi(d),$
We prove that
\begin{theorem}
\label{Thm-2}
Let $\chi$ be an odd primitive character modulo a prime number $p$. Let $f$ be a smooth function supported in $[X, 2X]\times[X, 2X]$ satisfying \eqref{0.1}. Then for $h \ll X^{1-\epsilon},$ if $p|h$ we have
 \begin{equation}
\label{Thm-2-eq}
\sum_{n-m=h}\tau_{\chi}(n)\tau_{\overline{\chi}}(m)f(n, m)= \text{Main term} + O(X^{1/2 +\theta + \epsilon}),
\end{equation}
where the Main term stated in the Equation \eqref{Main-th2}.\\
\label{Thm-3}\\
 For the sum of two squares, if $4|h$  we have
\begin{equation}
\label{Thm-3-eq}
\sum_{n-m=h}r(n)r(m)f(n ,m)= \text{Main term} + O(X^{1/2 +\theta + \epsilon}).
\end{equation}
where the main term comes from the setting $q=4$ in the Equation \eqref{Main-th2}.
\end{theorem}
 \noindent Note that \eqref{Thm-2-eq} improves, in the binary case, the error $O(X^{3/4+\epsilon})$ obtained by Heap \cite{W.H}.\\
Our method seems to be applicable to the shifted convolution sum of the divisor function and the a Fourier coefficient of cusp form of the full modular group and weight $k$. We expect following to hold
\begin{equation}
\label{Thm-4-eq}
\sum_{n-m=h}a(n)d(m)f(n, m)=  O(X^{1/2 + \epsilon}h^{\theta}).
\end{equation}
\noindent Next we look at more general shifted convolution sums.
\subsection{Quadratic divisor problem} We begin with recalling the fact that an application of the binary divisor problem is in bounding the moments of the zeta function. In this section we study a variation of the binary divisor problem that has applications in a wider and more complicated families of $L$-functions. Let $L(f, s)$ be the $L$-function attached to the $f$ in \eqref{cusp}, i.e.
\begin{equation}
L(s, f):= \sum_{n=1}^{\infty}a(n)n^{-s}.
\end{equation}
$L(f, s)$ satisfies a functional equation and by using the functional equation we obtain the convexity bound $L(s, f) \ll (k^2|s^2|D)^{\frac{1}{4}+\epsilon}.$
Note that the Lindel\"{o}f hypothesis asserts
\begin{equation}
\label{Lindelof}
L(s, f) \ll \big(k^2|s^2|D\big)^{\epsilon}.
\end{equation} In many applications it is suffices to replace the exponent $1/4$ by any smaller number. Such estimate is called  a subconvex bound that is also known as breaking the convexity bound. In order to break the convexity bound on $L(s, f)$, Duke, Friedlander and Iwaniec in \cite{DFI-2} needed an asymptotic with a good error term for $D_{f}(a, 1; h)$ where 
\begin{equation}
\label{0.2}
D_{f}(a, b; h):= \sum_{an-bm=h}d(n)d(m)f(an,bm).
\end{equation}
In \cite{DFI} they proved that if $f$ satisfies \eqref{0.1} then
\begin{equation}
\label{DFI-th}
D_{f}(a, b; h) = \text{\textit{ Main term}}(f, a, b) + E_{f}( a, b, h),
\end{equation}
where $E_{f}( a, b, h)= O(X^{3/4+\epsilon})$. Note that the main term has order of magnitude of $X/ab,$ thus the result is nontrivial as long as $ab <X^{1/4}$. In general, improving the error term or getting the error term of order $X^{1-\epsilon}/ab,$ appears to be an extremely hard problem which we discus in the next section. The purpose of this part of this article is to improve the error term when one of $a$ or $b$ equals $1$. We prove the following:
\begin{theorem}
\label{Thm-5}
Let Let $f$ be smooth function supported in $[X, 2X]\times[X, 2X]$ satisfying \eqref{0.1}. For $h \ll X^{1-\epsilon},$ We have
\begin{equation}
\label{Thm-5-eq}
D_{f}(a,1 ; h):= \sum_{an-m=h}d(n)d(m)f(an,m)= \text{\textit{ Main term}}(f, a, 1) + O(X^{1/2+ \theta+ \epsilon}),
\end{equation}
where the Main term stated in the Equation \eqref{main-thm5}.
\end{theorem}

\noindent Note that $\theta < 7/64 \approx 0.109375.$ This unconditionally improves the error term $O(X^{0.75+ \epsilon})$ of \cite{DFI} to $O(X^{0.6094}),$ and under the Ramanujan-Petersson conjecture to $O(X^{0.5+ \epsilon}).$\\
\\
\noindent In general to detect the condition $an-bm=h$ in the sum \eqref{0.2}, one needs to use some variant of the circle method. There are two major version of the circle method that can be used in shifted convolution problems. The $\delta$-method was invented by Duke, Friedlander and Iwaniec \cite{DFI-1}. They used it to solve the shifted convolution problem arising in breaking the  convexity bound on $L$-functions associated to holomorphic cusp forms. Their idea developed in many other papers to break the convexity bound on $L$-functions and the applications that will follow from breaking these bounds. (For more information see \cite{Kow-Mich-Vik}, \cite{Ph-Mich}). Another method used frequently in such problems is known as the Jutila circle method \cite{J}, which as mentioned earlier also has applications on shifted convolution sums for $GL(3)\times GL(2)$ \cite{Munshi}. In addition to these there is also a method using spectral theory which was suggested by Selberg \cite{Selberg} and made effective and general by Sarnak \cite{Sarnak}. However we cannot use any of these methods here because, in the $\delta$-method the inverses of $a$ and $b$ would enter in the Kloosterman sums. This would make it impossible to average the Kloosterman sums. As for the Jutila circle method and the spectral theory method we cannot use them since it would only work well with Fourier coefficients of modular forms. Here we use more elementary method that originally goes back to Heath-Brown and was used by Meurman in \cite{Meu}. 
\subsection{Generalized shifted divisor problem}
In the previous section we mentioned the connection between the quadratic divisor problem and sub-convexity bounds for families of $L$-functions. Breaking the convexity bound is a step forward towards the Lindel\"{o}f hypothesis (Equation \eqref{Lindelof}) for these $L$-functions.  Now let
\begin{equation}
M_k(T)= \int_{0}^{T} \big|\zeta\big(\tfrac{1}{2}+ it\big)\big|^{2k}dt
\end{equation}
be the $k$-th moment of the Riemann zeta function.  The Lindel\"{o}f hypothesis for the Riemann zeta function is equivalent to the statement that
\begin{equation}
\label{Lind}
 M_k(T)\ll T^{1+\epsilon},
\end{equation}
for all positive integers $k$ and all positive real numbers $\epsilon$. There is a close connection
between the generalized shifted divisor problem and the moments of the zeta function. Here we define the generalized shifted divisor problem, or $(k, l)$-shifted divisor problem, as finding non-trivial estimates for the sum
\begin{equation}
\label{gen-shift}
\sum_{n-m=h} d_k(n)d_l(m)f(n, m),
\end{equation}
where $d_k(n)= \# \{ (d_1 \cdots d_k)\in \mathbb{N}^k : d_1 \cdots d_k=n \},$ and $f$ is as \eqref{0.1}. There are conflicting conjectures regarding the size of the error term in  $(k, k)$-shifted divisor problem. 
\begin{conjecture}
Vinogradov \cite{A.V} conjectured \begin{equation}
\label{Vinog}
\sum_{n-m=h} d_k(n)d_k(m)f(n, m)= \text{\textit{Main term}}(f)+ O(X^{1-\frac{1}{k}}).
\end{equation}
\noindent Furthermore, Ivic \cite{Iv1} suggested that the $O$  term in \eqref{Vinog} should be replaced by $\Omega.$  Contrary to Ivic's conjecture, Conrey and Gonek's \cite{C.G} conjectured the following.
\label{CG-co}
\begin{equation}
\label{CG-con}
\sum_{n-m=h} d_k(n)d_k(m)f(n, m)= \text{\textit{Main term}}(f)+ O(X^{\frac{1}{2} +\epsilon}).
\end{equation}
\end{conjecture}
\noindent Note that their formulation of the conjecture stated in the case that $f$ is the indicator function of $[X, 2X]\times[X, 2X].$ However in practice we need to consider $f$ as in \eqref{0.1}. Ivic's conjecture in the case $k=2$ has proven by Motohashi \cite{Mo} and Szydlo improved the result and showed the error term in the case $k=2$ is $\Omega_{\pm}(X^{1/2}).$  \\

We mentioned some of the results for the $(2, 2)$-shifted divisor problem in the first part of the introduction. For $k,l >2 $ this problem remains unsolved and seems to be extremely hard. For the case $(k,l)=(3,2)$ an asymptotic formula was obtained by Hooley~\cite{Hoo}. For the case $(k,2)$ an asymptotic formula was derived by Linnik~\cite{Linik} using the dispersion method. Motohashi improved on Linnik's result by saving a power of $\log X$ in the error term. Power saving in the error term was obtained by Friedlander and Iwaniec~\cite{Fr-Iw} in the case $(k,l)=(3,2)$. They showed that there exists $\delta > 0$ such that the error term is smaller than $X^{1-\delta}$. Heath-Brown~\cite{Heath} showed that
$\delta= 1/102$ holds. \\
\\
\noindent Here we describe a bridge between the quadratic divisor problem and the $(k,l)$-shifted divisor problem. We explain this by means of the following lemma.
\begin{lemma}
\label{lem-con}
Let $f$ be a compactly supported function defined on $\mathbb{R}^2$. We have that
\begin{equation}
\label{average-quad}
\sum_{a, b} D_{f}(a, b; h)= \sum_{a, b} \sum_{an-bm=h}d(n)d(m)f(an,bm)= \sum_{m-n=h} d_3(n)\hspace{1 mm}d_3(m)\hspace{1 mm}f(n, m).
\end{equation}
In general for $k, l \geq 2$ we have
\begin{align}
\notag \sum_{\substack{a_i \\ 1 \leq i \leq k-1}} \hspace{1 mm} \sum_{\substack{b_j \\ 1 \leq j \leq l-1}}   \sum_{a_1 \cdots a_{k-1}n-b_1 \cdots b_{l-1}m=h} d(n)\hspace{1 mm}d(m&)\hspace{1 mm} f(a_1 \cdots a_{k-1}n,\hspace{1 mm} b_1 \cdots b_{l-1}m)\\ &= \sum_{m-n=h} d_k(n)\hspace{1 mm}d_l(m)\hspace{1 mm}f(n, m).
 \end{align}
\end{lemma}
This lemma essentially shows that by summing the $D_{f}(a, b; h)$ over $a$ and $b$ we can study the generalized shifted divisor problem. Therefore the error term in $(k,l)$-shifted divisor problem is the sum of the error terms in the quadratic divisor problem (sum of $E_{f}( a, b, h)$ over $a, b$ in the Equation \eqref{DFI-th}). This brings us to the following crucial question. \\
\\ \textbf{Question:} {What is the size of the error $E_{f}( a, b, h)$?}\\
\\
We may assume the following plausible assumptions:
\begin{enumerate}
\item The function $E$ as a function of $a, b$ oscillates mildly with respect to changes in $a, b$. This means that if $\parallel (a, b)-(c, d)\parallel_{_2}$ is small then $E_{f}( a, b, h)$ and $E(f, c, d, h)$ have about the same size.
 \item For $a, b \ll 1$, $E_{f}( a, b, h)=O(X^{1/2+\epsilon}).$
\item We assume that it possible to restrict the sum of $E_{f}( a, b, h)$ over $a, b$ in \eqref{average-quad} to the region $ab\ll X.$
\end{enumerate}
\noindent Using these heuristics we may conclude that either $E_{f}( a, b, h)= O(\sqrt X/ab)$ or $E_{f}( a, b, h)=O(\sqrt{X/ab}).$ Note that $E_{f}( a, b, h)= O(\sqrt X/ab)$ matches very well with Conrey and Gonek's conjecture for $(k, k)$-shifted divisor problem (equation \eqref{CG-con}). While assuming $E_{f}( a, b, h)= O(\sqrt{X/ab})$ only matches Vinogradov's conjecture for $(3, 3)$-shifted divisor problem. Moreover, to get the Vinogradov's conjecture (equation \eqref{Vinog}) for general $k$ one needs to assume an specific cancellations between $E_{f}( a, b, h)$ when we sum over $a, b$. This argument shows that the conjecture of Conrey and Gonek on the order of magnitude of the error terms in $(k,k)$-shifted divisor problem seems to be more accurate than the Vinogradov and Ivic's conjectures. \\
\\
\noindent  We conclude this section with pointing out the connection between the $(k, k)$-shifted divisor problem and the Lindel\"{o}f hypothesis for the Riemann zeta function. Ivic \cite{Iv1} has  shown that if \eqref{Vinog} holds for $k=3$ then \eqref{Lind} holds for  $k=3.$ Moreover, in \cite{Iv2} he proved that if we assume that the error term, in average over $h$ in \eqref{CG-con}, has square root cancellation, then the Lindel\"{o}f hypothesis for the Riemann zeta function would follow. \\
\\
\noindent \textbf{Structure of the paper and notations.} We will proceed first with introducing the $\delta$-Method and then using the Voronoi summation formulas to form a Kloosterman sums inside the the formulation derived with the $\delta$-Method. After that we will prove the necessary conditions that are needed for using the Kuznetsov formula in averaging the Kloosterman sums. We conclude the paper with treating the quadratic divisor problem with a different formulation but somehow similar with method used in the binary divisor problem. Note that throughout the paper we consider $h \ll X^{1-\epsilon}.$\\
\\
\noindent  \textbf{Kloosterman sum.} Let $m, n, q$ be natural numbers and $e(x)=e^{2\pi i x}$. The exponential sum 
\begin{equation}
 S(m ,n ;q)=\sum_{\substack{1\leq x <q \\ (x, q)=1}} e\big(\frac{mx+nx^{-1}}{q}\big).
\end{equation}
is called the Kloosterman sum. Weil \cite{A.W} proved that $$ S(m ,n ;q) \leq d(q) \sqrt{\gcd(m, n, q)}\sqrt{ q}.$$
Althought the Weil bound is optimal, on average the Kloosterman sum has a size about $q^{\epsilon}.$ This follows from the Kuznetsov formula. This is one of the major ideas used throughout the paper.\\
\noindent \textbf{Bessel functions.}  Throughout this article we make extensive use of the standard Bessel functions. They are defined as follows:
\begin{align*}
 &J_n(z)= \sum_{k=0}^{\infty} \frac{(-1)^{k}(z/2)^{2k+n}}{k!(n+k)!}, \\
& Y_n(z)=-\pi^{-1}\sum_{k=0}^{n-1} \frac{(n-k-1)!}{k!}(z/n)^{2k-n} \\ + & \pi^{-1}\sum_{k=0}^{\infty} \frac{(-1)^k(n-k-1)!}{k!}(2 \log(z/2)- \frac{\Gamma^{\prime}}{\Gamma}(k+1)- \frac{\Gamma^{\prime}}{\Gamma}(k+n+1)),\\ & K_n(z)=\frac{1}{2}\sum_{k=0}^{n-1} \frac{(-1)^k(n-k-1)!}{k!}(z/n)^{2k-n} \\ & + \frac{(-1)^{n-1}}{2}\sum_{k=0}^{\infty} \frac{(-1)^k(n-k-1)!}{k!}(2 \log(z/2)- \frac{\Gamma^{\prime}}{\Gamma}(k+1)- \frac{\Gamma^{\prime}}{\Gamma}(k+n+1)).
\end{align*}
Moreover we use the following properties.  $\displaystyle{\big(z^v Y_{v}(z)\big)^{\prime}=z^v Y_{v-1}(z)},$  $\displaystyle{\big(z^v K_{v}(z)\big)^{\prime}=-z^v K_{v-1}(z)},$ and $\displaystyle{\big(z^v J_{v}(z)\big)^{\prime}=z^v J_{v-1}(z)}.$ We also use the following bounds from \cite{Kow-Mich-Vik}(Lemma C.2). For $z>0$ and $k \geq 0$
\begin{align}
\label{Bess-deriv-est}
& \big(\frac{z}{1+z}\big)^{i}Y^{(i)}_0(z) \ll \frac{(1+ |\log z|)}{(1+z)^{1/2}}, \\ &
\label{Bess-deriv-est-k}
\big(\frac{z}{1+z}\big)^{i}K^{(i)}_0(z) \ll \frac{e^{-z}(1+ |\log z|)}{(1+z)^{1/2}}.
\end{align}
For further properties of Bessel functions see \cite{Bess}.
\section{\textbf{$\delta$-method}}
In this section we follow \cite{DFI} to introduce and set up the $\delta$-method. Let $Q>0$ and $w(u)$ be an even, smooth, compactly supported function on $Q \leq |u|\leq 2Q$ and
\begin{equation}
 \label{der-w}
 w^{(i)}(u) \ll \frac{1}{Q^{i+1}}, \text{        } \text{        } \text{        }\sum_{q=1}^{\infty}w(q)=1.
\end{equation}
The $\delta$ function is defined on $\mathbb{Z}$ by $\delta(0)=1$ and $\delta(m)=0$ for $m\neq 0.$ The $\delta$-method is a decomposition of the $\delta$ function in terms of additive characters $e(\cdot)$ on rational numbers. More precisely we have:
\begin{equation}
\label{def-delta}
\delta(m)= \sum_{q=1}^{\infty} \hspace{2 mm} \sum^{\hspace{ 5mm} *}_{d \text{ (mod }q\text{)}}e\big(\frac{md}{q}\big)\Delta_q(m),
\end{equation}
 where $$\Delta_q(m)=\sum_{r=1}^{\infty}\frac{w(qr)-w(\frac{m}{qr})}{qr}.$$
Let $f(x, y)$ be a differentiable function supported in $ [X, 2X]\times[X, 2X] $ satisfying \eqref{0.1}.
Let $\phi$ be a smooth function supported on $[-X, X]$ with the property that $\phi^{(i)} \ll X^{-i}.$ By applying \eqref{def-delta} to detect the condition $m-n=h$ in \eqref{Thm-1-eq} we have
\begin{equation}
\label{delt-thm-1}
\sum_{n-m=h}d(n)d(m)f(n, m)= \sum_{q<Q} \sum_{d \text{ mod }q}^{ \hspace{ 5mm} *} e\big(\frac{-hd}{q} \big) \sum_{m, n}d(m)d(n) e\big(\frac{dn-dm}{q} \big)E(n, m, q),
\end{equation}
where
\begin{equation}
\label{E}
E(x,y, q)=f(x, y)\phi(x-y-h)\Delta_q(x-y-h).
\end{equation}
 For the left hand side of \eqref{Thm-2-eq}, \eqref{Thm-3-eq} and \eqref{Thm-4-eq} we have similar formula.
\section{\textbf{Voronoi summation formulas}}
Let $f(n)$ be an arithmetic function, let $q$ be an integer, and let $g(n)$ be a compactly supported function on $\mathbb{R^+}.$ For $(d, q)=1$ we have
\begin{equation}
\label{eq2.1}
\sum_{n} f(n)e\big(\frac{nd}{q}\big)g(n)= \sum_{a \text{ mod }q} e\big(\frac{a}{q}\big) \sum_{n \equiv \overline{d} a \text{ mod }q} f(n)g(n).
\end{equation}
Now if for $a \neq 0,$ $f(n)$ has some sort of well distribution modulo $q$ one can study the main term and the error term in \eqref{eq2.1}. The general Voronoi summation formula studies the sum of the type \eqref{eq2.1} for certain sequences. The idea started with Voronoi in \cite{Vrn}. Here we state the Voronoi summation formula for  $d(n), \tau_{\chi}(n), r(n)$ and $a(n).$
\begin{lemma}
\label{Vor-Lemma}
Let $g(x)$ be a smooth, compactly supported function on $\mathbb{R}^{+}$ and let $(d, q)=1.$ We have
\begin{equation}
\sum_{n=1}^{\infty} d(n)e\big(\frac{nd}{q} \big)g(n)= \frac{1}{q}  \int_{0}^{\infty} (\log x + 2\gamma - \log q)g(x)dx + \sum_{i=1}^{2}\sum_{n=1}^{\infty}e\big(\frac{\pm n\overline{d}}{q} \big)g_{i}(n),
\end{equation}

where
$$g_{1}(n)= -\frac{2 \pi}{q} \int_{0}^{\infty} g(x) Y_{0}\big(\frac{4 \pi \sqrt{nx}}{q}\big) dx$$
$$g_{2}(n)= \frac{4}{q} \int_{0}^{\infty} g(x) K_{0}\big(\frac{4 \pi \sqrt{nx}}{q}\big)dx.$$
If $(c, q)=1$, for $\tau_{\chi}(n)$, where $\chi$ is an odd Dirichlet character modulo $c$, we have
\begin{align}
\label{(p, q)=1}
\sum_{n=1}^{\infty} \tau_{\chi}(n)e\big(\frac{nd}{q} & \big)g(n)= \frac{\chi(q)}{q} L(1, \chi)  \int_{0}^{\infty} g(x)dx  \\ & \notag -2 \pi \frac{\chi(q)}{q}\frac{\tau(\chi)}{c}\sum_{n=1}^{\infty}\tau_{\chi}(n)e\big(\frac{- n\overline{dc}}{q} \big)\int_{0}^{\infty} g(x) J_{0}\big(\frac{4 \pi \sqrt{nx}}{\sqrt{c}q}\big)dx,
\end{align}
if $c|q$ we have
\begin{align}
\label{p|q}
\sum_{n=1}^{\infty} \tau_{\chi}(n)e\big(\frac{nd}{q} & \big)g(n)= \frac{\chi(\overline d)}{q} \tau(\chi) L(1, \overline \chi)  \int_{0}^{\infty} g(x)dx  \\ & \notag -2 \pi i \frac{\chi(\overline d)}{q}\sum_{n=1}^{\infty}\tau_{\chi}(n)e\big(\frac{- n\overline{d}}{q} \big)\int_{0}^{\infty} g(x)J_{0}\big(\frac{4 \pi \sqrt{nx}}{q}\big)dx,
\end{align}
and if $c=4$ and $q \equiv 2 \text{ mod } 4$ we have
\begin{align}
\sum_{n=1}^{\infty} r(n)e\big(\frac{nd}{q}  \big)g(n)=  -2 \pi i \frac{\chi(\overline d)}{q}\sum_{n=1}^{\infty}r^{*}(n)e\big(\frac{- n\overline{d}}{2q} \big)\int_{0}^{\infty} g(x)J_{0}\big(\frac{4 \pi \sqrt{nx}}{q\sqrt 2}\big)dx,
\end{align}
where
\begin{equation}
\label{r^*}
r^{*}(n)= \sum_{m_1m_2=n} \chi_4(m_1)\big(1- (-1)^{m_1}\big)
\end{equation}
Finally, for Fourier coefficients of weight $k$ cusp form we have
\begin{align}
\sum_{n=1}^{\infty} & a(n)e\big(\frac{nd}{q}  \big)g(n)=  \frac{-2 \pi i^k}{q}\sum_{n=1}^{\infty}a(n)e\big(\frac{- n\overline{d}}{q} \big)\int_{0}^{\infty} g(x)J_{k-1}\big(\frac{4 \pi \sqrt{nx}}{q}\big)dx.
\end{align}
Here $Y_0, K_0$ and $J_k$ are Bessel functions.
\end{lemma}
\noindent The formula for $d(n)$ is due to Jutila \cite{Jut-d(n)}. The formula for $a(n)$ and $\tau_{\chi}(n)$ in the case $(c, q)=1$ and $c|q$ can be found in Chapter 4 of \cite{Iw-kow}. Here we would give a proof for the case $c=4$ and $q \equiv 2 \text{ mod } 4$.
\begin{proof}[\bf{ Proof of Lemma \ref{Vor-Lemma}.}] Let $\chi_4$ be  a non principal odd character modulo $4$. We have
\begin{equation}
\label{r(n)-2}
\sum_{n}r(n) e\big(\frac{nd}{q}\big)g(n)= \sum_{n}\sum_{m_1m_2=n} \chi_{4}(m_1)e\big(\frac{m_1m_2d}{q}\big)g(m_1m_2).
\end{equation}
We set $m_1=2n_1q+u_1$ and $m_2=n_2q+u_2$. Since $q \equiv 2 \hspace{2 mm} (\text{ mod }4),$ with this choice of $m_1$ we have $\chi_{4}(m_1)=\chi_{4}(u_1)$ and therefore \eqref{r(n)-2} is equal to
\begin{equation}
\label{r(n)-2-1}
 \sum_{\substack{u_1 \text{ mod } 2q \\ u_2 \text{ mod } q}} \sum_{n_1, n_2} \chi_{4}(u_1)e\big(\frac{u_1u_2d}{q}\big)g\big((u_1+2qn_1)(u_2+qn_2)\big)
\end{equation}
We apply the Poisson summation formula (Equation (4.24) of \cite{Iw-kow}) to the sum over $n_1, n_2.$ Therefore, \eqref{r(n)-2-1} is equal to
\begin{equation}
\label{r(n)-2-2}
 \frac{1}{2q^2}\sum_{m_1, m_2} \sum_{\substack{u_1 \text{ mod } 2q \\ u_2 \text{ mod } q}}  \chi_{4}(u_1)e\big(\frac{u_1u_2d}{q}+ \frac{m_2u_2}{q}+ \frac{m_1u_1}{2q}\big)\hat{g}\big(\frac{m_1}{\sqrt 2 q}, \frac{m_2}{\sqrt 2 q}\big),
\end{equation}
where $$\hat{g}\big(\frac{m_1}{\sqrt 2 q}, \frac{m_2}{\sqrt 2 q}\big)=\int_{0}^{\infty} \int_{0}^{\infty} g(xy)e\big(\frac{-m_1x}{\sqrt 2 q}\big) e\big(\frac{-m_2y}{\sqrt 2 q}\big)dx dy.$$
Note that $\sqrt 2$ in the denominator comes from the change of variable inside the above integral. Now the sum over $u_2$ inside  \eqref{r(n)-2-2} is zero unless $u_1 \equiv \overline{d}m_2 \hspace{2 mm } (q),$ in which case the sum is equal to $q$. Since we considered $u_1$ modulo $2q$ we have only choices $ \overline{d}m_2,  \overline{d}m_2 +q$ for $u_1.$ Considering this \eqref{r(n)-2-2} is equal to
\begin{equation}
\label{r(n)-2-3}
 \frac{1}{2q}\sum_{m_1, m_2}  \bigg( \chi_{4}(\overline{d}m_2)e\big(\frac{m_1m_2\overline{d}}{2q}\big) + (-1)^{m_1}\chi_{4}(\overline{d}m_2+q)e\big(\frac{m_1m_2\overline{d}}{2q} \big) \bigg)\hat{g}\big(\frac{m_1}{\sqrt 2 q}, \frac{m_2}{\sqrt 2 q}\big)
\end{equation}
The rest of the proof follows exactly the proof of Theorem 4.14 in \cite{Iw-kow}. Note that in the case  $q \equiv 2 \hspace{2 mm} (4)$  we do not have a main term because the main term comes from setting $m_1$ or $m_2$ equal to zero. For $m_1=0$ we get $\chi_{4}(\overline{d}m_2)+\chi_{4}(\overline{d}m_2+q)$ inside the parenthesis in \eqref{r(n)-2-3}, which is equal to zero since  $q \equiv 2 \hspace{2 mm} (4)$. For $m_2=0$ we have both $\chi_{4}(\overline{d}m_2)$ and $\chi_{4}(\overline{d}m_2+q)$ are equal to zero.
\end{proof}
 \noindent Next we apply the Voronoi summation formula to \eqref{delt-thm-1} and corresponding formulas for \eqref{Thm-2-eq}, \eqref{Thm-3-eq} and \eqref{Thm-4-eq}.
\section{\textbf{Toward Kloosterman sums}}
 In this part we apply the Voronoi summation formula (Lemma \ref{Vor-Lemma}) to the sums we derived from the $\delta$-method. This will lead to the Kloosterman sums inside our formula for the error terms. Our final aim is to average the Kloosterman sums and obtain sharp estimates for the error terms. For $\tau_{\chi}(n)$ we will work out the formula in detail. For the shifted convolution of $r(n)$ in  \eqref{Thm-3-eq}, and the shifted convolution of $d(n)$ and $a(n)$ in \eqref{Thm-4-eq}, we will give the final formula. As for the divisor function, we will write the result for $d(n)$ using Equation (24) in \cite{DFI}. We consider the general case $a, b$ not necessarily equal to $1$ to explain why the method cannot be applied to the quadratic divisor problem.
\subsection{Formula for $\tau_{\chi} $} By using the $\delta$-method we have
\begin{align}
\label{delt-thm-2}
\sum_{n-m=h}\tau_{\chi}(n)& \overline{\tau_{\chi}(m)}f(n, m)= \\ & \notag \sum_{q<Q} \sum_{d \text{ mod }q}^{ \hspace{ 5mm} *} e\big(\frac{-hd}{q} \big) \sum_{n, m}\tau_{\chi}(n)\overline{\tau_{\chi}(m)} e\big(\frac{dn-dm}{q} \big)E(n, m, q).
\end{align}
Recall that $E(\cdot, \cdot, \cdot)$ is defined in \eqref{E}. First we split the sum over $q$ into two cases: $(p, q)=1$ and $p|q.$ For $(p, q)=1$ we apply \eqref{(p, q)=1} first to the sum over $n$ and we end up with two terms. Then we apply \eqref{(p, q)=1} to the sum over $m$ and we get two other terms. Consequently we have
\begin{align}
\label{delt-thm-2-ex}
& \sum_{\substack{q<Q \\ (p, q)=1}} \notag \sum_{d \text{ mod }q}^{ \hspace{ 5mm} *} e\big(\frac{-hd}{q} \big) \sum_{m, n}\tau_{\chi}(n)\overline{\tau_{\chi}(m)} e\big(\frac{dn-dm}{q} \big)E(m, n, q) \\ & =\sum_{\substack{q<Q\\ (p, q)=1}}  \sum_{d \text{ mod }q}^{ \hspace{ 5mm} *} e\big(\frac{-hd}{q} \big)\frac{|\chi({q})|^2}{q^2} L(1, \chi)^2 \int_{0}^{\infty} \int_{0}^{\infty} E(x, y, q)dx dy  \\ & \notag  + \sum_{\substack{q<Q\\ (p, q)=1}} \sum_{d \text{ mod }q}^{ \hspace{ 5mm} *}\boldsymbol{ e\big(\frac{-hd}{q} \big)} \Bigg ( 2 \pi \frac{|\chi({q})|^2}{q^2} \frac{\tau(\chi)}{p}L(1, \chi) \sum_{n=1}^{\infty} \tau_{\chi}(n)\boldsymbol{e\big(\frac{-n\overline{dp}}{q} \big)}  \int_{0}^{\infty} \int_{0}^{\infty} E(x, y, q) J_{0}\big(\frac{4\pi \sqrt{nx}}{q\sqrt p}\big)dx dy \\ & \notag - 2 \pi \frac{|\chi({q})|^2}{q^2} \frac{\tau(\chi)}{p} L(1, {\chi}) \sum_{m=1}^{\infty} \overline{\tau_{\chi}(m)}\boldsymbol{e\big(\frac{-m\overline{dp}}{q} \big)}  \int_{0}^{\infty} \int_{0}^{\infty} E(x, y, q) J_{0}\big(\frac{4\pi \sqrt{my}}{q\sqrt p}\big)dx dy   +  (2 \pi )^2\frac{|\chi({q})|^2}{q^2} \big(\frac{\tau(\chi)}{p}\big)^2 \\ & \notag \times \sum_{n, m=1}^{\infty} \tau_{\chi}(n)\overline{\tau_{\chi}(m)} \boldsymbol{e\big(\frac{-(n+m)\overline{dp}}{q} \big)}  \int_{0}^{\infty} \int_{0}^{\infty} E(x, y, q) J_{0}\big(\frac{4\pi \sqrt{nx}}{q\sqrt p}\big) J_{0}\big(\frac{4\pi \sqrt{my}}{q\sqrt p}\big)dx dy \Bigg ).
\end{align}
\noindent Now since we assumed that $p|h$ we write $h=h^\prime p$ and the terms in bold will form our Kloosterman sums. For $p|q$ we apply \eqref{p|q} in Lemma \ref{Vor-Lemma} once to the sum over $m$ and once to the sum over $n$. Therefore, when  $p|q$ \eqref{delt-thm-2} is equal to
\begin{align}
\label{delt-thm-2-exx}
&\sum_{\substack{q<Q\\ p|q}}  \sum_{d \text{ mod }q}^{ \hspace{ 5mm} *} e\big(\frac{-hd}{q} \big)\frac{|\chi(\overline{d})|^2}{q^2} \tau^{2}(\chi)L(1, \overline{\chi})^2 \int_{0}^{\infty} \int_{0}^{\infty} E(x, y, q)dx dy  \\ & \notag  + \sum_{\substack{q<Q\\ p|q}} \sum_{d \text{ mod }q}^{ \hspace{ 5mm} *}\boldsymbol{ e\big(\frac{-hd}{q} \big)} \Bigg ( 2 \pi i\frac{|\chi(\overline{d})|^2}{q^2} \tau(\chi)L(1, \overline{\chi}) \sum_{n=1}^{\infty} \tau_{\chi}(n)\boldsymbol{e\big(\frac{-n\overline{d}}{q} \big)}  \int_{0}^{\infty} \int_{0}^{\infty} E(x, y, q) J_{0}\big(\frac{4\pi \sqrt{nx}}{q}\big)dx dy \\ & \notag + 2 \pi i\frac{|\chi(\overline{d})|^2}{q^2} \tau(\chi)L(1, \overline{\chi}) \sum_{m=1}^{\infty} \overline{\tau_{\chi}(m)}\boldsymbol{e\big(\frac{-m\overline{d}}{q} \big)}  \int_{0}^{\infty} \int_{0}^{\infty} E(x, y, q) J_{0}\big(\frac{4\pi \sqrt{my}}{q}\big)dx dy \\ & \notag +  (2 \pi i)^2\frac{|\chi(\overline{d})|^2}{q^2}  \sum_{n, m=1}^{\infty} \tau_{\chi}(n)\overline{\tau_{\chi}(m)} \boldsymbol{e\big(\frac{-(n+m)\overline{d}}{q} \big)}  \int_{0}^{\infty} \int_{0}^{\infty} E(x, y, q) J_{0}\big(\frac{4\pi \sqrt{nx}}{q}\big) J_{0}\big(\frac{4\pi \sqrt{my}}{q}\big)dx dy \Bigg ).
\end{align}
Note that the terms in bold will form our Kloosterman sums.
\subsection{Formula for the sum of two squares} For $r(n)$ we have a formula similar to \eqref{delt-thm-2} with  $\tau_{\chi}(\cdot)$ replaced by $r(\cdot)$. Here we have to split the summation over $q$ to three cases: $4|q$, $(4, q)=1$ and $q \equiv 2 \hspace{2 mm} (4).$ For the first two cases the final formula would be the same as \eqref{delt-thm-2-ex} and \eqref{delt-thm-2-exx} with $p=4$ and $\tau_{\chi}(\cdot)=r(\cdot)$. The case we need to work out is  $q \equiv 2 \hspace{2 mm} (4).$ Let $r^{*}$ be as \eqref{r^*}. Then the corresponding formula to \eqref{delt-thm-2} for $r(n)$ is
\begin{align}
\label{delt-thm-2-exxx}
&\sum_{\substack{q<Q\\ q \equiv 2 \hspace{2 mm} (4)}}  \sum_{d \text{ mod }q}^{ \hspace{ 5mm} *} \boldsymbol{e\big(\frac{-hd}{q} \big)}\bigg(\frac{|\chi(\overline{d})|^2}{q^2} (2 \pi i)^2 \\ & \notag \sum_{n, m=1}^{\infty} r^{*}(n)r^{*}(m) \boldsymbol{e\big(\frac{-(n+m)\overline{d}}{2q} \big)}  \int_{0}^{\infty} \int_{0}^{\infty} E(x, y, q) J_{0}\big(\frac{4\pi \sqrt{nx}}{\sqrt 2  q}\big) J_{0}\big(\frac{4\pi \sqrt{my}}{\sqrt 2 q}\big)dx dy\bigg).
\end{align}
Note that since $r^{*}(n)=0$ for even $n$, we can write the sum over $m, n$ in \eqref{delt-thm-2-exx} in terms of odd $m, n$. Therefore $(n+m)/2$ is an integer and with this we will have our Kloosterman sums.
\subsection{Formula for the divisor function}
For the shifted convolution sum of $d(n)$, i.e. equation \eqref{0.2}, using  \cite[Equation (24)]{DFI} we have
\begin{align}
\label{1.2}
D_{f}(a, &b; h)=\sum_{q<Q} \frac{(ab,q)}{q^2} \Bigg(  S (h, 0; q)I +  \notag \sum_{n=1}^{\infty}d(n)S(h, (a, q)\overline{a_q}n; q  )I_{a}(n, q)  \\  &  + \sum_{m=1}^{\infty}d(m)S(h, -(b, q)\overline{b_q}m; q  )I_{b}(m, q)  \\ \notag & +
\sum_{m,n=1}^{\infty}d(n)d(m)S(h, (a, q)\overline{a_q}n-(b, q)\overline{b_q}m; q  )I_{ab}(n, m, q)+ \divideontimes \divideontimes \divideontimes \divideontimes \divideontimes \Bigg),
\end{align}
where $a_q=a/(a, q)$ and $\overline{a_q}\text{ is the inverse of } a_q \text{ modulo } q$ and
\begin{align}
\label{1.3}
I_{ab}(&n, m, q)=  4\pi^{2}\int_{0}^{\infty} \int_{0}^{\infty} Y_0\big(\frac{4\pi(a, q)\sqrt{mx}}{q}\big)Y_0\big(\frac{4\pi(b, q)\sqrt{ny}}{q}\big)E(x, y, q)dxdy,
\end{align}
and
\begin{align}
\label{1.4}
I_{b}(&n, q)=  -2\pi\int_{0}^{\infty} \int_{0}^{\infty} (\log(ax)- \lambda_{a, q})Y_0\big(\frac{4\pi(b, q)\sqrt{ny}}{q}\big)E(x, y, q)dxdy,
\end{align}
and $ \lambda_{a, q}=2\gamma + \log\frac{aq^2}{(a, q)^2}.$ The $\divideontimes \divideontimes \divideontimes \divideontimes \divideontimes$ means we have $5$ more terms involving a $K_0$-Bessel function. The main term for $D_{f}(a, b; h)$ comes from the contribution of the sum over $q$  of $\mathcal{S} (h, 0; q)I$ in \eqref{1.2}. As for the error term have two types of error terms :
$$  \mathcal{E}_1(a, b)= \sum_{q<Q} \frac{(ab,q)}{q^2} \sum_{n=1}^{\infty}d(n){S}(h, (a, q)\overline{a_q}n; q  )I_{a}(n, q),$$
and
$$\mathcal{E}_2(a, b)= \sum_{q<Q} \frac{(ab,q)}{q^2} \sum_{m,n=1}^{\infty}d(n)d(m){S}(h, (a, q)\overline{a_q}n-(b, q)\overline{b_q}m; q  )I_{ab}(n, m, q),$$
and our aim is to show that $\mathcal{E}_1(1, 1), \mathcal{E}_2(1, 1) \ll X^{1/2+ \epsilon}h^{\theta}.$
\begin{remark}
In \cite{DFI} Equation (24) there is a typo.Instead of $(a, q)\overline{a_q}$ they have $\overline{a}$, inside the Kloosterman sums. We would also like to emphasize that we cannot prove our result for $D_f(a, b; h)$ for $a, b\neq 1$ because the term $\overline{a_q}$ enters into the Kloosterman sums. Ultimately, this will make the averaging impossible. The advantage of Jutila circle method \cite{J} is  that the sum over $q$ can be restricted to the multiples of $ab$, while in the $\delta$-method the sum over $q$ runs over all integers less than $Q$. However, it seems difficult to apply Jutila's circle method to the divisor function as it does not have square root cancellation while the sum of the coefficients coming from holomorphic or cusp forms has square root cancellation, see \cite{Bl}.
\end{remark}

\noindent Here we just need to deal with the error term arising from $I_{ab}(n, m, q)$. The error terms arising from $I_a(m, q)$ and $I_b(n, q)$ can be handled  similar to $I_{ab}(n, m, q)$. Throughout the proof we will make comments on the similarity between $I_{ab}(n, m, q),$  $I_{a}(n, q)$  and $I_{b}(m, q)$. The other errors in  \eqref{1.2}, corresponding to ($\divideontimes \divideontimes \divideontimes \divideontimes \divideontimes$), that arise from the $K_0$-Bessel function, can be handled by the similarity between the $K_0$ and  $Y_0$-Bessel functions. We set $a,b=1$ and write $I(n, m, q)$ in a place of $I_{ab}(n, m, q)$. \\

\section{\textbf{Bounding $I(n, m, q)$}}
In this section we will do the necessary adjustments in order to be able to use results regarding averaging Kloosterman sums. The main difficulty in proving the fact that $I(m, n)$ oscillates mildly in respect to $q$, comes from small $q$. In \cite{DFI} the parameter $Q$, in the $\delta$-method, is equal to  $2\sqrt{X}$. If we use the same choice of $Q$ and follow the  method in \cite[Equation (30)]{DFI} for $q\ll 1$ we get the bound $I(n, m, q) \ll \sqrt{X}$, while we need $I(n, m, q) \ll X^{\varepsilon}$. We overcome this difficulty by changing the parameter $Q$ from $\sqrt X$ to  $X^{1/2+ \varepsilon}.$ As a result we have to consider a wider range for the sum over $q$ in \eqref{1.2}. However, the faster rate of decay of the partial derivatives of $w$ in the $\delta$-method will help us to show that $I(n, m, q)$ is very small for $\displaystyle{q<X^{\frac{1}{2}-\varepsilon}}$.
Let $I(n, m, q)$ be as \eqref{1.3} with $a=  b=1,$ $$I(n, m, q)=  4\pi^{2}\int_{0}^{\infty} \int_{0}^{\infty} Y_0\big(\frac{4\pi\sqrt{mx}}{q}\big)Y_0\big(\frac{4\pi\sqrt{ny}}{q}\big)E(x, y, q)dxdy.$$
We will prove the following lemmas to show that the contribution of small $q$'s in \eqref{1.2} are negligible. We will also find upper bounds for the range of the sum over $n$ in \eqref{1.2}. We will consider $I(n, m, q)$ with a $Y_0-$Bessel function, but the same lemmas are valid with replacing $Y_0$ with $J_0-$Bessel function. This is true because the properties of the $Y_0-$Bessel function that we will use in the proof are that $\displaystyle{(z^v Y_{v}(z))^{\prime}=z^v Y_{v-1}(z)}$ and $Y_{v}(z) \ll 1/ \sqrt{z}.$ We also have the same properties for the $J_0-$Bessel functions: $\displaystyle{(z^v J_{v}(z))^{\prime}=z^v J_{v-1}(z)}$ and $J_{v}(z) \ll 1/ \sqrt{z}.$
\begin{lemma}
\label{lemma-q}
For $\varepsilon>0$ there exist $i, j \in \mathbb{N}$ such that, for  $q< X^{1/2 - \varepsilon}$ we have $$I(n, m, q) \ll {m^{-i/2-1/4}n^{-j/2-1/4}X^{-1}}$$ and $$I(n, q) \ll {n^{-j/2}X^{-1}}.$$
\end{lemma}
\begin{proof}
We begin with a change of variable in $I(n, m, q).$ Setting $\displaystyle{u=\frac{4\pi\sqrt{mx}}{q}}$ and $v=\displaystyle{\frac{4\pi\sqrt{ny}}{q}}$ in the expression for $I(n, m, q)$ yields
\begin{align}
\label{2.1}
I(n, m, q)=\frac{4\pi^{2}q^4}{(4\pi)^4 mn}\int_{0}^{\infty} \int_{0}^{\infty} uY_0(u)vY_0(v)E\bigg(\frac{u^2q^2}{(4\pi)^{2}m}, \frac{v^2q^2}{(4\pi)^{2}n}, q\bigg)dudv.
\end{align}
By employing the recursive formula $\displaystyle{(z^v Y_{v}(z))^{\prime}=z^v Y_{v-1}(z)}$ and integration by parts in \eqref{2.1}  we have
\begin{align}
\label{2.2}
I(n,& m, q) \asymp\frac{q^{2(i+j+2)}}{m^{i+1}n^{j+1}}  \int_{0}^{\infty} \int_{0}^{\infty} u^{i+1}Y_i(u)v^{j+1}Y_j(v)E^{(i, j, 0)}\bigg(\frac{u^2q^2}{(4\pi)^{2}m}, \frac{v^2q^2}{(4\pi)^{2}n}, q\bigg)dudv.
\end{align}
Similarly, by integrating by parts in \eqref{1.4}, for $I(n, q)$ we deduce that
\begin{align}
\label{2.2}
I(n, q)\asymp\frac{q^{2(j+1)}}{n^{j+1}}  \times\int_{0}^{\infty} \int_{0}^{\infty}\notag (\log x- \lambda_q)v^{j+1}Y_j(v)E^{(0, j, 0)}\bigg(x, \frac{v^2q^2}{(4\pi)^{2}n}, q\bigg)dxdv.
\end{align}
Here we need to estimate the partial derivatives of  $E(x,y, q).$ Recall that $E(x,y, q)=f(x, y)\phi(x-y-h)\Delta_q(x-y-h),$ and for the partial derivatives of $E$ we have
\begin{equation}
\label{2.3}
E^{(i, j, 0)}:= \frac{\partial^{i+j}}{\partial x^{i} \partial y^{j}} E(x, y, q)= 
\sum_{\substack{r, r^{\prime}, s, s^{\prime} \geq 0\\ r+r^{\prime}=i \\ s+ s^{\prime}=j}} c_{r,r^{\prime}, s, s^{\prime}}(f\phi)^{(r,s)}\Delta^{(r^{\prime}, s^{\prime})}_q.
\end{equation}
 For the partial derivative of $f$ and $\Delta_q,$ we have $(f\phi)^{(r,s)} \ll \frac{1}{X^{(r+s)}}$ and $\Delta^{(r^{\prime}, s^{\prime})}_q \ll \frac{1}{(qQ)^{(r^{\prime}+ s^{\prime}+1)}}.$ Now since $\displaystyle{q<X^{1/2-\varepsilon}}$ we have $qQ<X,$ so the major term in \eqref{2.3} is $\Delta^{(i, j)}_q$. Therefore we have $E^{(i, j, 0)} \ll {(qQ)^{-i-j-1}}.$ We will apply the latter bound for $E^{(i, j, 0)}$ together with the bound  $Y_{i}(u) \ll \frac{1}{\sqrt{u}}$ in \eqref{2.2} to get
\begin{align}
 I(n, m, q) & \ll  \frac{q^{2(i+j+2)}}{m^{i+1}n^{j+1}(qQ)^{i+j+1}} \int_{\  \frac{\sqrt{mX}}{q}}^{_{\frac{\sqrt{2mX}}{q}}} \int_{\ \frac{ \sqrt{nX}}{q}}^{_{\frac{\sqrt{2nX}}{q}}} u^{i+\frac{1}{2}}v^{j+\frac{1}{2}}du dv  \\ & \notag \ll_{i,j} \frac{q^{2(i+j+2)}X^{(i+j+3)/2}m^{i/2+3/4}n^{j/2+3/4}}{m^{i+1}n^{j+1}(qQ)^{i+j+1}q^{i+j+3}}  \ll_{i,j} \frac{X^{(i+j+3)/2}}{m^{i/2+1/4}n^{j/2+1/4}Q^{i+j+1}}.
 \end{align}
A similar argument for $I(n, q)$ yields: $$I(n, q) \ll  \frac{X^{(j+4)/2}}{n^{j/2}Q^{j+1}}. $$ Now, using $Q=X^{1/2 +\varepsilon}$ and $\displaystyle{j=\big \lfloor\frac{3}{\varepsilon}}\big \rfloor$ completes the proof.
\end{proof}
\noindent  The following lemma will provide the bound for the sum over $m, n$ in Equation (24) in \cite{DFI}
\begin{lemma}
\label{lemma-r}
For $ X^{1/2 - \varepsilon}< q < X^{1/2 + \varepsilon}$, the contribution of $m,n> X^{3\varepsilon}$ in \eqref{1.2} is negligible.\\
\end{lemma}
\begin{proof}
Since $ X^{1/2 - \varepsilon}< q < X^{1/2 + \varepsilon}$, we have $\frac{1}{qQ} > \frac{1}{X}.$ Therefore $E^{(i, j, 0)} \ll {X^{-i-j-1}}. $  We are using same bounds as Lemma \ref{lemma-q} in \eqref{2.2} and consequently we have
\begin{align*}
 I(n, m, q)\notag &  \notag \ll_{i,j} \frac{q^{2(i+j+2)}X^{(i+j+3)/2}m^{i/2+3/4}n^{j/2+3/4}}{m^{i+1}n^{j+1}X^{i+j+1}q^{i+j+3}}  \\ & \ll_{i,j} \frac{q^{i+j+1}X^{(i+j+3)/2}}{m^{i/2+1/4}n^{j/2+1/4}X^{i+j+1}}.
 \end{align*}
 Now using $q<X^{1/2 + \varepsilon}$ we get $$I(n, m, q)  \ll_{i,j} \frac{X^{\varepsilon(i+j+1)+1}}{m^{i/2+1/4}n^{j/2+1/4}}. $$Similarly  for $I(n, q)$ we have :
 $$I(n, q)  \ll_{j} \frac{X^{\varepsilon(j+1)+1}}{n^{j/2}}.$$ And therefore by taking $\displaystyle{j=\big \lfloor\frac{3}{\varepsilon}}\big \rfloor$ we have $$\sum_{m,n > X^{3\varepsilon}}d(n)d(m)I(n, m, q) \ll_{i, j} \frac{X}{X^{{\varepsilon(i+j-3)/2}}} \ll \frac{1}{\sqrt{X}}.$$ The same bound for the sum over $I(n, q)$ holds. Using this in \eqref{1.2} combined with the trivial bound on the Kloosterman sums gives us the error term of order $O(X^{-1/2+ \epsilon}).$ This shows the the contribution of $m,n> X^{3\varepsilon}$ is negligible and we only need to consider the sum over $m, n$ in \eqref{1.2} up to $X^{3\varepsilon}.$ This finishes the proof of the Lemma.
\end{proof}
Basically  Lemma \ref{lemma-q} and \ref{lemma-r} show that $$\mathcal{E}_2(1, 1)= \sum_{q<X^{1/2 - \epsilon}< q< X^{1/2 + \epsilon}}  \frac{1}{q^2} \sum_{ m,n < X^{3\epsilon} }^{\infty}d(n)d(m)S(h, n-m; q)I(n, m, q)+ O(\frac{1}{\sqrt X}), $$
and a similar argument for  $\mathcal{E}_1(1, 1)$: $$\mathcal{E}_1(1, 1) +\sum_{q<X^{1/2 - \epsilon}< q< X^{1/2 + \epsilon}}  \frac{1}{q^2} \sum_{ m,n < X^{3\epsilon} }^{\infty} d(n){S}(h, n; q  )I(n, q) + O(\frac{1}{\sqrt X}).$$
\section{\textbf{Averaging the Kloosterman Sums}}
In this part we state the lemmas that we will need in  averaging the Kloosterman sums. These results were derived by an application of the Kuznetsov formula. The first lemma  is due to  Deshouillers and Iwaniec \cite{J}.  This will be used when we average the Kloosterman sums over all moduli. \\
\begin{lemma}
\label{Lemma-J}
Let $m \geq 1,$ $P>0,$ $Q>0,$ and let $g(x, y)$ be a function of class $C^4$ with support on
$[P, 2P]\times [Q, 2Q]$ satisfying
$$\frac{\partial^{i+j} }{\partial q^{i} \partial r^{j}}g(x, y) \ll \frac{1}{P^i Q^j} \text{  for } 0 \leq i, j \leq 2.$$
Then for any complex numbers $a_p$ we have
\begin{align}
 \sum_{P< p <2P} \sum_{Q< q < 2Q} a_p g(p, q)  S(h, \pm p, q)  \ll (\sqrt{h}+Q) P^{1/2}
h^{\theta}\bigg(\sum_{p} |a_{p}|^2\bigg)^{1/2}(hPQ)^{\epsilon}.
\end{align}
\end{lemma}
\noindent The second lemma \cite[Proposition 3.5.]{Bl} is useful when the averaging is over multiples of an integer.
\begin{lemma}
\label{Lemma-Blo}
With notations of Lemma \ref{Lemma-J} and for $N>0$ we have
\begin{align}
 \sum_{P< p <2P} \sum_{\substack{Q< q < 2Q \\ N|q}} a_p g(p, q) & S(h, \pm p, q)   \\ & \notag \ll Q \bigg(\sum_{p} |a_{p}|^2\bigg)^{1/2}
\big(1+ \frac{hP}{Q^2} +\frac{P}{N}\big)^{1/2} h^{\theta} \Big(1+ \big(\frac{Q^2}{hP}\big)^{\theta}\Big)(hPQ)^{\epsilon}.
\end{align}
\end{lemma}
\section{\textbf{Proof of Theorems}}
\noindent In this section we prove Theorems \ref{Thm-1} and \ref{Thm-2}. An important part of the proofs of these theorems is to show that the functions that are attached to the Kloosterman sums oscillate mildly. We will show this for the function $I(n, m, q)$, defined in \eqref{1.3}. The other functions are similar to this case. For fixed $n$, $I(m+r,m, q)$ is a function of $q$ and $r$ so we set $I(m+r,m, q) :=\mathcal{I}(q, r)$. In order to apply Lemma \ref{Lemma-J} we need to show that  ${\mathcal{I}(q, r)}/{q^2}$ oscillates mildly with respect to $q, r$. 
\begin{lemma}
\label{lemma-final}
Let $X^{1/2 -\varepsilon}<\mathcal{Q}<X^{1/2 +\varepsilon}$ and $\mathcal{R}<X^{3\varepsilon}$. Then for $\mathcal{Q}<q<2 \mathcal{Q}$ and $\mathcal{R}<r< 2\mathcal{R}$ we have
$$\frac{\partial^{i+j} }{\partial q^{i} \partial r^{j}}\Big(\frac{1}{X^{136\varepsilon}}\frac{\mathcal{I}(q, r)}{q^2}\Big) \ll \frac{1}{\mathcal{Q}^i \mathcal{R}^j}.$$
\end{lemma}
\begin{proof}
By \eqref{2.1} it follows that 
\begin{align}
\label{q^'}
\frac{\partial^{} }{\partial q}\frac{\mathcal{I}(q, r)}{q^2}  \asymp \frac{1}{ n(n+r)}\int_{0}^{\infty} \int_{0}^{\infty} uY_0(u)vY_0(v)\frac{\partial^{} }{\partial q}q^2E(\frac{u^2q^2}{(4\pi)^{2}(n+r)}, \frac{v^2q^2}{(4\pi)^{2}n}, q)dudv.
\end{align} By the chain rule for multi-variable functions
\begin{align}
\label{3.2}
\frac{\partial^{} }{\partial q} q^2E=2qE+\frac{2q^3u^2}{(4\pi)^{2}(n+r)}E^{(1,0, 0)}+\frac{2q^3v^2}{(4\pi)^{2}n}E^{(0,1, 0)}+q^2E^{(0, 0, 1)}.
\end{align}
Considering the range of $q$ in Lemma, we use these bounds $E<1/X$, $E^{(1,0 , 0)}<1/X^2$ and $E^{(0, 1, 0)}<1/X^2$. For $E^{(0 ,0, 1)}$ we need to estimate
\begin{align}
\label{3.3}
& \displaystyle{\frac{\partial^{} }{\partial q}\Delta_q(x-y-h)=\sum_{r=1}^{\infty} \frac{\partial^{} }{\partial q} \Bigg(\frac{w(qr)-w(\frac{x-y-h}{qr})}{qr}}\Bigg)=  -\frac{\Delta_q(x-y-h)}{q} \\ & \notag +\sum_{r}\frac{w^{\prime}(qr)}{q}+\frac{(x-y-h)w^{\prime}(x-y-h/qr)}{q^3r^2}.
\end{align}
 By Lemma 2 in \cite{DFI} we have $ w^{i}(u) \ll \frac{1}{Q^{i+1}},$  and $\Delta_{q}(u) \ll \frac{1}{qQ}$. Therefore, each term  in \eqref{3.3} is bounded by $\frac{1}{qQ^2}.$ Plugging in these bounds into \eqref{3.2} and considering the range of $q$ we have $$\frac{\partial^{} }{\partial q} q^2E \ll (1+u^2 +v^2)\frac{X^{3\varepsilon}}{\sqrt{X}}.$$
We use the above in \eqref{q^'} and by taking to account that since $X<\frac{u^2q^2}{(4\pi)^{2}(n+r)}<2X$ the range in the integral for $u, v$ is $0<u, v<X^{5\varepsilon},$ we have $$\frac{\partial^{} }{\partial q}\frac{\mathcal{I}(q, r)}{q^2} \ll \frac{X^{34\varepsilon}}{X^{{\frac{1}{2}}+\varepsilon}} \ll \frac{1}{\mathcal{Q}^{1-\epsilon}}.$$
For the second derivative with respect to $q$ we apply the same method to each term in \eqref{3.2} and use the similar bound on the derivatives of $E$. For the derivatives with respect to $r$ we have
 \begin{align*}
\notag & \frac{\partial^{} }{\partial r}\frac{\mathcal{I}(q, r)}{q^2}   =\frac{4\pi^{2}}{(4\pi)^4 n(n+r)}\int_{0}^{\infty} \int_{0}^{\infty} uY_0\big(u\big)vY_0\big(v\big)\frac{\partial^{} }{\partial r}q^2E(\frac{u^2q^2}{(4\pi)^{2}(n+r)}, \frac{v^2q^2}{(4\pi)^{2}n}, q)dudv.
\end{align*}
Also we have
$$\Big|\frac{\partial^{} }{\partial r}q^2E(\frac{u^2q^2}{(4\pi)^{2}(n+r)}, \frac{v^2q^2}{(4\pi)^{2}n}, q)\Big|=\Big| \frac{u^2q^4}{(4\pi)^{2}(n+r)^2}E^{(1, 0, 0)} \Big| \ll \frac{u^2X^{4\varepsilon}}{r^2},$$
and therefore $$\frac{\partial^{} }{\partial r}\frac{\mathcal{I}(q, r)}{q^2} \ll \frac{X^{34\varepsilon}}{\mathcal{R}^2}.$$
Similar method we use for second derivative in respect to $r$ and derivative in respect to $q, r$. This finishes the proof of Lemma.
\end{proof}
Now we need to apply Lemma \ref{Lemma-J} to ${X^{-136\varepsilon}}{\mathcal{I}(q, r)}{q^{-2}}$. In order to do that we need to put the support of the function in dyadic intervals. Here we use Harcos's treatment \cite{Ha}. Let $\rho$ be a smooth function whose support lies in $[1, 2]$ and satisfies the following identity for $x>0$:
\begin{equation}
\label{rho}
\sum_{k=-\infty}^{\infty}\rho(2^{-k/2}x)=1.
\end{equation}
We write $$\frac{1}{X^{136\varepsilon}}\frac{\mathcal{I}(q, r)}{q^2}=\sum_{k,l=-\infty}^{\infty}\mathcal{I}_{k,l}(q, r),$$
where $$\mathcal{I}_{k,l}(q, r)=\displaystyle{\frac{1}{X^{136\varepsilon}}\frac{\mathcal{I}(q, r)}{q^2}\rho(\frac{q}{2^{k/2}Q})\rho(\frac{r}{2^{l/2}R})}$$ and $Q=X^{1/2 + \varepsilon}$ and $R=X^{3\varepsilon}$. The support of $\mathcal{I}_{k,l}(q, r)$ is $[2^{k/2}Q, 2^{k/2+\hspace{1 mm}1}Q]\times [2^{l/2}R, 2^{l/2+ \hspace{1 mm}1}R]$.
\noindent Note that we just need to use Lemma \ref{Lemma-J} in the range of $X^{1/2-\epsilon}<q<X^{1/2+\epsilon}$ and $r<R$. For $q, r$ outside of this range we estimate the sum \eqref{1.2} trivially using Lemma \ref{lemma-q} and \ref{lemma-r}.

\begin{proof}[\bf Proof of Theorem \ref{Thm-1}.]
Take $\varepsilon=\epsilon/134$. We apply Lemma \ref{Lemma-J} to $\mathcal{I}_{k,l}(q, r)$ for $-4\varepsilon \frac{\log X}{\log 2}\leq k\leq 0$ and $-6\varepsilon \frac{\log X}{\log 2} \leq l \leq 0$. For the derivatives of $\mathcal{I}_{k,l}(q, r)$ we need to have bounds on the derivatives of $\rho(\frac{q}{2^{k/2}Q})\rho(\frac{r}{2^{l/2}R})$ for which we have
\begin{equation}
\frac{\partial^{} }{\partial q} \rho(\frac{q}{2^{k/2}Q})\rho(\frac{r}{2^{l/2}R})\ll \frac{1}{2^{k/2}Q} \ll \frac{X^{4\varepsilon}}{Q}
\end{equation}
The derivative with respect to $r$ has a similar bound. Using Lemma \ref{lemma-final} and the above we have that $\mathcal{I}_{k,l}(q, r)$ satisfies the condition of Lemma \ref{Lemma-J}. Therefore to have a upper bound on the error terms arising from \eqref{1.2}, we apply Lemma \ref{Lemma-J} with $m=h$ and $P=R$ and $a_p= d(n)d(n+p)$ to $\mathcal{I}_{k,l}(q, r)$. This will take care of the error term $\mathcal{E}_2(1, 1)$ arising from $$\sum_{q<Q} \frac{1}{q^2} \Bigg( \sum_{m,n=1}^{\infty}d(n)d(m)\mathcal{S}(h, m-n ; q  )I(n, m, q)\Bigg).$$ For the error terms $\mathcal{E}_1(1, 1)$ we follow a similar method to show $I(n, q)$ in \eqref{1.4} oscillates mildly with respect to $q$ and $n$. The only difference with $I(n, m, q)$ is that instead of one of the Bessel functions in \eqref{1.3} we have a $\log q$ term coming from $ \lambda_{1, q}$. The derivative of $\log q$ has the desired decay with respect to $q$. This will finish the proof of theorem
\end{proof}
\begin{proof}[\bf Proof of Theorem \ref{Thm-2}.]\textbf{Main term.} The main term here comes from Equations \eqref{delt-thm-2-ex} and \eqref{delt-thm-2-exx}. We combine these with the argument of the section 6 of \cite{DFI} and we have the main term is
\begin{align}
\label{Main-th2}
& \sum_{\substack{q<Q\\ (p, q)=1}} S(h, 0; q) \frac{|\chi({q})|^2}{q^2} L(1, \chi)^2 \int_{0}^{\infty} f(x, x-h)\phi(h)dx
\\ & \notag +\sum_{\substack{q<Q\\ p|q}}  S(h, 0; q)\frac{\tau^{2}(\chi)}{q^2} L(1, \overline{\chi})^2 \int_{0}^{\infty} f(x, x-h)\phi(h)dx
\end{align}
\textbf{Error term.} The difference with the proof of Theorem \ref{Thm-1} is that we split the sum over $q$  in \eqref{delt-thm-2-ex}  and \eqref{delt-thm-2-exx} into two cases: $(p, q)=1$ and $p|q$. Recall that $h^{\prime}=h/p.$ For $(p, q)=1$ we need to deal with averaging the Kloosterman sums of the form $S(-h^{\prime}, -(n+m); q)$ over $q$ and $m,n$: $$\sum_{\substack{q<Q \\ (p,q)=1}} \frac{1}{q^2} \Bigg( \sum_{m,n=1}^{\infty}\tau_{\chi}(n)\overline{\tau_{\chi}(m)}S(-h^{\prime}, -(n+m);\hspace{1 mm} q)I(n, m, q)\Bigg).$$ Note that $I(n, m, q)$ would be the same as Theorem \ref{Thm-1} if we change $Y_0$ to $J_0$. For $p|q$ we average Kloosterman sums of the form $S(-h, -(n+m); \hspace{1 mm} q)$. For the sum over $q$ with the condition $p|q$ we use Lemma \ref{Lemma-Blo} and we have the error term $O(X^{1/2 + \theta + \epsilon}).$ For the sum over $q$ with the condition $(p,q)=1,$ first we add the following terms  $$\sum_{\substack{q<Q \\ p|q}} \frac{1}{q^2} \Bigg( \sum_{m,n=1}^{\infty}\tau_{\chi}(n)\overline{\tau_{\chi}(m)}S(-h^{\prime}, -(n+m);\hspace{1 mm} q)I(n, m, q)\Bigg)$$ and then subtract them. By adding this we get a sum over all $q,$ so we use Lemma \ref{Lemma-J} and we get the error therm $O(X^{1/2+\epsilon}h^{\theta}.)$ For the terms that we had subtracted we use Lemma \ref{Lemma-Blo} and we get the error term $O(X^{1/2 + \theta + \epsilon}).$
The final error will be $O(X^{1/2 + \theta + \epsilon})$ as in \eqref{Thm-2-eq}. This finishes the proof of the first part of the Theorem. \\
\\
\noindent Now to show the second part of the theorem regarding the sum of two squares note that the main term in this case is the same as \eqref{Main-th2} with setting $q=4.$ For the error term the proof is also very close to the proof of of  \eqref{Thm-2-eq}, with only minor modification. Here the sum over $q$ is divided to three cases: $(q, 4)=1,$ $4|q$ and $q\equiv2 \hspace{1 mm}(4)$ and each involves a Kloosterman sums with different arguments. For $(q, 4)=1,$ we add the sums over even $q$'s and subtract them. We use Lemmas \ref{Lemma-J} and \ref{Lemma-Blo} respectively. The error term is  $O(X^{1/2 + \theta + \epsilon})$. For $4|q$ we use Lemma \ref{Lemma-Blo} and the error term is $O(X^{1/2 + \theta + \epsilon})$. Finally for  $q\equiv2 \hspace{1 mm}(4),$ we add the sums over over $q$'s such that $4|q$ and subtract them. We use Lemma \ref{Lemma-Blo} twice, once with sum over even $q$'s and once with the sum over $q$'s such that $4|q$. This finishes the proof of the Theorem.
\end{proof}
\noindent In the next chapter we will prove Theorem \ref{Thm-5}.
\section{\textbf{Quadratic divisor problem}}
In this section we use a version of Dirichlet's hyperbola method to write the divisor function $d(n)$ in terms of a summation of a weight function. Our analysis of the left hand side of \eqref{Thm-5-eq} follows the argument in \cite{Bl-Ha-Mi}. Let $\omega$ be a smooth function such that $\omega(x) = 1$ on $[0, 1]$ and $\omega(x) = 0$ on $[2, \infty)$. For $n<Q$ we have
\begin{equation}
d(n)= \sum_{ \delta | n } \omega\big ( \frac{\delta}{\sqrt Q}\big)\big(2- \omega\big ( \frac{n}{\delta \sqrt Q}\big)\big).
\end{equation}
Thus
\begin{align*}
& \sum_{an-m=h}d(n)d(m)f(an, m)= \sum_{n=1}^{\infty}d(n)d(an-h)f(an, an-h) \\ & = \sum_{\delta =1}^{\infty}\omega\big(\frac{\delta}{Q}\big) \sum_{\delta | n}d(an-h)f(an, an-h)\bigg(2- \omega\big(\frac{n}{\delta Q}\big)\bigg) \\  & =  \sum_{\delta =1}^{\infty}\omega\big(\frac{\delta}{Q}\big) \sum_{m \equiv -h (a\delta)} d(m)f(m-h, m)\bigg(2- \omega\big(\frac{m+h}{a\delta Q}\big)\bigg)
\end{align*}
Using Corollary 4.12. of \cite{Iw-kow} for the innermost sum we have
\begin{align}
\label{6.2}
&\sum_{an-m=h}d(n)d(m)f(an, m)\notag= \\ &\sum_{q=1}^{\infty}\frac{(a, q) S(h, 0; q)}{q^2} \int_{0}^{\infty} \big(\log\big(\frac{x}{q^2})+ 2\lambda \big)K_{(a, q), q}(x)f(ax, ax-h)dx \\ & -2\pi \sum_{q=1}^{\infty} \frac{(a, q)}{q^2}  \notag \sum_{n=1}^{\infty} d(n) S(h, n; q) \int_{0}^{\infty} Y_0 \big(\frac{4\pi\sqrt{n(ax-h)}}{q}\big) K_{(a, q), q}(x)f(ax, ax-h)dx \\ & -2\pi
\notag \sum_{q=1}^{\infty} \frac{(a, q)}{q^2} \sum_{n=1}^{\infty} d(n) S(h, n; q) \int_{0}^{\infty}  K_0 \big(\frac{4\pi\sqrt{n(ax-h)}}{q}\big) K_{(a, q), q}(x)f(ax, ax-h)dx
\end{align}
where $$K_{r, q}(x)= \sum_{\delta}^{\infty} \frac{1}{\delta} \omega\big ( \frac{q\delta}{r\sqrt Q}\big)\big(2- \omega\big ( \frac{rx}{\delta q \sqrt Q}\big)\big).$$
Note that here since the support of $f$ is $[X, 2X]\times [X, 2X]$, we can take $Q= 2X/a,$ also using the definition we have that $K_{r, q}(x)=0$ for $q> 2r \sqrt{Q}.$ Similar to the proof of Theorem \ref{Thm-1} we need to show that the function attached to the Kloosterman sums in \eqref{6.2} oscillates mildly with respect with $q$ and $n$. In order to do this there is some difficulty in dealing with the function $K_{(a, q), q}(x)$.  Note that if $a|q$, $K_{(a, q), q}(x)$ might be positive while $K_{(a, q+1), q+1}(x)$ equals to zero and thus this function does not oscillate mildly with respect to $q$. Therefore we need to average over $q$'s such that $(a, q)$ is fixed. As in the case of the binary divisor problem we will show that the contribution from small $q$'s is negligible. The integral in the second sum in \eqref{6.2} equals 
\begin{align}
\label{var-change}
\frac{1}{a}\int_{0}^{\infty} Y_0 \big(\frac{4\pi\sqrt{n(x-h)}}{q}\big) K_{(a, q), q}\big(\frac{x}{a}\big)f(x, x-h)dx,
\end{align}
By the variable change $ax \rightarrow x$.  In order to prove our result we need to estimate the second and third sum in \eqref{6.2}. Using \eqref{var-change} we have that the second sum in \eqref{6.2} equals 
\begin{equation}
\label{S_2}
\uptheta:=\sum_{q=1}^{\infty} \frac{(a, q)}{aq^2}  \sum_{n=1}^{\infty} d(n) S(h, n; q) \int_{0}^{\infty}  Y_0 \big(\frac{4\pi\sqrt{n(x-h)}}{q}\big) K_{(a, q), q}\big(\frac{x}{a}\big)f(x, x-h)dx
\end{equation}
For the rest of the paper we focus on estimating this sum since the third sum in \eqref{6.2} shall satify the same bound and can be handled in a similar way. Let $(a, q)=d$, this condition is  equivalent to $(a/d, q/d)=1$ and we detect this with $\sum_{\sigma | (\frac{a}{d}, \frac{q}{d})} \mu(\sigma).$ Using this the outer sum in \eqref{S_2} simplifies to
\begin{equation}
\sum_{d|a}\sum_{\substack{q \\ \hspace{2 mm}(a,q)=d}}= \sum_{d|a}\sum_{\sigma | \frac{a}{d}}\mu(\sigma) \sum_{\substack{q \\ \hspace{2 mm}\sigma d|q}}
\end{equation}
and hence
\begin{equation}
\label{inc-ex-1}
 \uptheta= \sum_{d|a}\sum_{\sigma | \frac{a}{d}}\mu(\sigma) \sum_{\substack{q \\ \hspace{2 mm}\sigma d|q}} \uptheta_{a, h}(\sigma),
\end{equation}
where
\begin{equation}
\label{6.5}
 \uptheta_{a, h}(\sigma)=   \frac{d}{a} \sum_{\substack{q=1 \\ \sigma d|q }}^{\infty} \frac{1}{q^2} \sum_{n=1}^{\infty} d(n) S(h, n; q)I(n, q, d),
\end{equation}
and \begin{equation}
\label{Inq}
I(n, q, d) :=\int_{0}^{\infty} -2 \pi Y_0 \big(\frac{4\pi\sqrt{n(x-h)}}{q}\big) K_{d, q}(x/a)f(x, x-h)dx.
\end{equation}
Our aim is to show that
\begin{equation}
 \label{fin}
\uptheta_{a, h}(\sigma) \ll \frac{d}{a} X^{1/2+\theta} .
\end{equation} Putting this in \eqref{inc-ex-1} gives 
\begin{equation}
\label{finn}
 \uptheta \ll  X^{1/2+\theta} \sum_{d|a}  \frac{2^{\omega(a/d)}d}{a} \ll X^{1/2+\theta+\epsilon}.
\end{equation}
In order to prove \eqref{fin} we will divide the range of the summation over $n, q$ to three cases:
\begin{enumerate}
 \item  $q<X^{1/2-\epsilon}$ and $n\geq 1.$\\
\item $X^{1/2-\epsilon} \leq q \leq \sqrt{aX}$ and $n \gg q^2/X^{1-3\epsilon}.$\\
\item $X^{1/2-\epsilon} \leq q \leq \sqrt{aX}$ and $n \ll q^2/X^{1-3\epsilon}.$
\end{enumerate}
We estimate cases 1, 2, by using Lemmas \ref{lem6.1} and \ref{lem-6.2} and considering the trivial bound on the Kloosterman sum. For the case 3 wee need to estimate by using the application of the Kuznetsove formula in averaging the Kloosterman sums. To proceed we prove the following Lemmas for $I(n, q, d)$ to show that the contribution of small $q$'s are negligible
\begin{lemma}
\label{lem6.1}
For $q<X^{1/2 - \epsilon},$ we have that
$$I(n, q, d) \ll \frac{1}{n^2 X}. $$
\end{lemma}
\begin{proof}
Since this argument is similar to the proof of Lemma \ref{lemma-q}, we just give a sketch of the proof. First we use the fact \cite[Equation (2.30)]{Bl-Ha-Mi} that
\begin{equation}
\label{deriv-k}
\frac{\partial^{i+j}}{\partial x^{i} \partial q^{j}} K_{r, q}(x) \ll_{i, j} \frac{ \log Q}{X^{i}q^j}.
\end{equation}
By using the change of variable $u={4\pi\sqrt{n(x-h)}}/{q}$ and by  integration by parts $i$ times 
\begin{equation}
\label{Lemma61}
I(n, q, d) \asymp \frac{q^{2(i+1)}}{n^{i+1}}  \int_{0}^{\infty} u^{i+1}Y_i\big(u\big)\frac{\partial^{i}}{\partial u^{i}  }\big(f(\frac{u^2q^2}{n(4\pi)^2}+ h)K_{d, q}(\frac{u^2q^2}{an(4\pi)^2}+ \frac{h}{a})\big)du.
\end{equation}
Now by applying \eqref{deriv-k} for derivatives of $K$ and \eqref{0.1} for derivatives of $f,$ and by using the fact that  $q<X^{1/2 - \epsilon},$ we conclude the proof of the Lemma.
\end{proof}
\begin{lemma}
\label{lem-6.2}
Let $\mathcal{Q}> X^{1/2-\epsilon}$. Then for $q \in [\mathcal{Q}, 2\mathcal{Q}],$ we have that
\begin{equation}
\label{7.10}
\sum_{n > \mathcal{Q}^2/{X^{1-3\epsilon}}} d(n)|I(n, q, d)| \ll 1.
\end{equation}
\end{lemma}
\begin{proof}
Considering the range $q \in [\mathcal{Q}, 2\mathcal{Q}],$ in \eqref{Lemma61} and similar to the proof of Lemma \ref{lem6.1}, we use \eqref{deriv-k} and \eqref{0.1} on derivatives of $f$ and $K$. Finally by taking $i$ large enough we will have \eqref{7.10}.
\end{proof}
\noindent This lemma shows that we only need to consider the contribution of $n < \mathcal{Q}^2/{X^{1-3\epsilon}}$ in \eqref{6.5}. Now we need to show that the derivatives of $I(n, q, d)/q^2,$ satisfy the conditions of Lemma \ref{Lemma-Blo}. Note that by using a smooth partition of unity
similar to the proof of Theorem \ref{Thm-1} we break the support of $I(n, q, d)$ to dyadic intervals. The largest error term comes from $d\sqrt{2X/a}< q<2d\sqrt{X/a},$ and $\displaystyle{{d^2 X^{\epsilon}}{a^{-1}} < n< 2{d^2 X^{\epsilon}}{a^{-1}}.}$ Also note that $K_{d, q}=0$ for $q>2d\sqrt{X/a}.$
\begin{lemma}
\label{lemma-6.3}
Let $\mathcal{Q}> X^{1/2-\epsilon}$ and $q \in [\mathcal{Q}, 2\mathcal{Q}],$ and $n \in [N, 2N]$ then for $ 0 \leq i, j \leq 2$
\begin{equation}
\label{deriv-I}
\frac{\partial^{i+j}}{\partial q^{i} \partial n^{j}} \big( \frac{\mathcal{Q}^{2-\epsilon}}{X}\frac{I(n, q, d)}{q^2} \big ) \ll_{i, j} \frac{1}{\mathcal{Q}^{i}N^j}.
\end{equation}
\end{lemma}
\begin{proof}
We differentiate once with respect to $q$ and once with respect to $n$. We  state the necessary bounds on functions in the integrant \eqref{Inq}. The derivative in respect to $n$
is \begin{equation}
\label{eq-dev-1}
\frac{\partial}{ \partial n} I(n, q, d) =\int_{0}^{\infty} -  \frac{1}{nq^2}\frac{4\pi^2\sqrt{n(x-h)}}{q}Y^{\prime}_0 \big(\frac{4\pi\sqrt{n(x-h)}}{q}\big) K_{d, q}(x/a)f(x, x-h)dx
\end{equation}
By Lemma \ref{lem-6.2} we may assume $n \leq q^2/x^{1 - 3\epsilon}$. Therefore for
$z= {4\pi\sqrt{n(x-h)}}{/q},$ since $x \in [X, 2X]$ we have $z \ll X^{\epsilon}$. In order to use \eqref{Bess-deriv-est} we need to multiply the integral \eqref{eq-dev-1} with $1/(1+z)$ and since $z \ll X^{\epsilon}$ this would at most augment it by $X^{\epsilon}$. Now we pull out $1/n$ from  \eqref{eq-dev-1} and we use  $q \in [\mathcal{Q}, 2\mathcal{Q}],$ and the fact that $f$ is supported in $[X, 2X] \times [X, 2X]$ and $K \ll \log X$ to get \eqref{deriv-I} for $(i=0, j=1).$  If we differentiate $I(n, q, d)$ with respect to $q$ using \eqref{deriv-k} we obtain  \eqref{deriv-I} for $(i=1, j=0)$ exactly similar to the case $(i=0, j=1).$ Now we differentiate \eqref{eq-dev-1} with respect to $q$ to obtain \eqref{deriv-I} for $(i=1, j=1).$
\begin{align}
\label{eq-dev-2}
 \frac{\partial}{ \partial q \partial n}& I(n, q, d)  = \\ &  \notag \frac{1}{n}\int_{0}^{\infty} - \frac{\partial}{ \partial q}\bigg( \frac{1}{q^2}\frac{4\pi^2\sqrt{(x-h)}}{q}Y^{\prime}_0 \big(\frac{4\pi\sqrt{n(x-h)}}{q}\big) K_{d, q}(x/a)f(x, x-h)dx\bigg).
\end{align}
All the terms with $q$ in the denominator and also $K_{d, q}$ would obviously give us the $1/q$ saving that we need. We just treat the term with derivative of Bessel function.
$$\frac{\partial}{ \partial q}Y^{\prime}_0 \big(\frac{4\pi\sqrt{n(x-h)}}{q}\big) = \frac{1}{q}\bigg(\frac{-4\pi\sqrt{n(x-h)}}{q}Y^{(2)}_0 \big(\frac{4\pi\sqrt{n(x-h)}}{q}\big)\bigg),$$
which for bounding this we use \eqref{Bess-deriv-est} and the fact that $z \ll X^\epsilon.$ For the cases $(i=2, j=1)$, $(i=1, j=2)$ and $(i=2, j=2)$ the proof is similar to the case $(i=1, j=1).$ This finishes the proof of the Lemma.
\end{proof}
\begin{proof}[\bf Proof of Theorem \ref{Thm-5}] Here the main term is:
\begin{equation}
\label{main-thm5}
\sum_{q=1}^{\infty}\frac{(a, q) S(h, 0; q)}{q^2} \int_{0}^{\infty} \big(\log\big(\frac{x}{q^2})+ 2\lambda \big)K_{(a, q), q}(x)f(ax, ax-h)dx
\end{equation}
For the error term, note that in order to use Lemma \ref{Lemma-Blo} we need to use the smooth partition of unity to put the support of $I(n, q, d)/q^2$ in dyadic intervals. In order to do this let $\rho$ be the same as the proof of Theorem \ref{Thm-1} and $\rho$ satisfies \eqref{rho} and $R=d\sqrt{2X/a}$ and $S=d^2 X^{\epsilon}/{a},$ we write  ${q^{-2}}{I(n, q, d)}=\sum_{k,l=-\infty}^{\infty}I_{k,l}(n, q, d),$
where
\begin{equation}
\label{partition}
 I_{k,l}(n, q, d)=q^{-2}{I(n, q, d)}\rho(\frac{q}{2^{k/2}R})\rho(\frac{n}{2^{l/2}S})
\end{equation}
  The support of $I_{k, l}$ is $[2^{k/2}R, 2^{k/2+\hspace{1 mm}1}R]\times [2^{l/2}S, 2^{l/2+ \hspace{1 mm}1}S].$ Now since  $K_{d, q}=0$ for $q>2d\sqrt{X/a}$ and the fact that $I_{k, l}$ is supported on $2^{k/2} d\sqrt{2X/a} \leq q \leq 2^{k/2 +1} \hspace{1 mm} d\sqrt{2X/a}$, we conclude that $k \leq 0.$ Also we have $|k| \ll \log X.$
To continue with the proof we are returning to our range separation for the summation over $q, n$ in \eqref{6.5}:
\begin{enumerate}
 \item  $q<X^{1/2-\epsilon}$ and $n\geq 1.$ For this range using Lemma \ref{lem6.1} we have $I_{k, l}(n, q, d) \ll I(n, q, d) \ll {n^{-2} X^{-1}}$ and consequently
\begin{align*}
\sum_{\substack{q=1 \\ d\sigma|q }}^{X^{1/2- \epsilon}}\sum_{k, l}  & \hspace{1 mm}  \sum_{n=1}^{\infty} d(n)\frac{I_{k ,l}(n, q, d)}{q^2} S(h, n; q)  \\ & \ll \sum_{\substack{q=1 \\ d\sigma|q }}^{X^{1/2- \epsilon}} \hspace{1 mm}  \sum_{n=1}^{\infty} d(n)\frac{I(n, q, d)}{q^2} S(h, n; q) \ll \frac{1}{X}\hspace{1 mm}\sum_{\substack{q=1 \\ d\sigma|q }}^{X^{1/2- \epsilon}} \hspace{1 mm} \sum_{n=1}^{\infty} \frac{d(n)}{qn^2} \ll X^{\epsilon-1}.
\end{align*}
\item $X^{1/2-\epsilon} \leq q \leq d(X/a)^{1/2}$ and $n \gg q^2X^{3\epsilon-1}.$ First note that if $d/\sqrt a < X^{-\epsilon}$ then since $K_{d, q}=0$ for $q>2d(X/a)^{1/2},$ we fall into the first range. Now by using Lemma \ref{lem-6.2} we have
\begin{align*}
\sum_{k, l} & \sum_{\substack{ X^{1/2- \epsilon} \leq q \leq {d(X/a)^{1/2}} \\ d\sigma|q  }} \hspace{2 mm}  \sum_{n \gg q^2X^{3\epsilon-1}} d(n)\frac{I_{k, l}(n, q, d)}{q^2} S(h, n; q) \\ &
\sum_{\substack{ X^{1/2- \epsilon} \leq q \leq {d(X/a)^{1/2}} \\ d\sigma|q  }} \hspace{2 mm}  \sum_{n \gg q^2X^{3\epsilon-1}} d(n)\frac{I(n, q, d)}{q^2} S(h, n; q) \ll \sum_{\substack{ X^{1/2- \epsilon} \leq q \leq {d(X/a)^{1/2}}  }} \hspace{2 mm} \frac{1}{q} \ll 1.
\end{align*}
\item $X^{1/2-\epsilon} \leq q \leq \sqrt{aX}$ and $n \ll q^2X^{3\epsilon-1}.$ For this range we need to apply Lemma \ref{Lemma-Blo}.  By using \eqref{partition} we beak the support of $I$ into dyadic intervals and for the current range we have to deal with $I_{k, l}$ where $$ \frac{2}{\log 2}\big( \log \big(\frac{\sqrt a}{\sqrt{2}d X^{\epsilon}}\big) -1\big) \leq  k \leq 0.$$
and $$ \frac{2}{\log 2} \log\big(\frac{a}{d^2X^{\epsilon}}\big) \leq l \leq \frac{2}{\log 2} \log\big(\frac{d^2q^2}{aX^{1-2\epsilon}}\big)-3.$$
Our aim is to handled the following sum for $k,l$ in the above range
\begin{equation*}
\sum_{k, l}\sum_{\substack{q=1 \\ d\sigma|q }}^{\infty}  \sum_{n=1}^{\infty} d(n)\frac{I_{k, l}(n, q, d)}{q^2} S(h, n; q)
\end{equation*}
    Each of the $I_{k, l}$ in the above range can be handled using Lemma \ref{Lemma-Blo}. Here we only consider the range $R< q<2R,$ and $S < n< 2S$ i.e $k=l=1,$ for which we have the biggest error term. Now since the volume of the box that $k, l$ take their values in, is bounded by $(\log X)^2$ we have the final error term is bounded by $(\log X)^2$ times the error that comes from $k=l=1$. Lemma \ref{lemma-6.3} enable us to average the Kloosterman sums above by employing Lemma \ref{Lemma-Blo} and also save a factor $X/\mathcal{Q}^{2-\epsilon}$. Thus by setting $a_p=d(p)$, $Q=R$ and $P=S$ in Lemma \ref{lemma-6.3} we have the above is bounded by
    $$\frac{a}{d^2} d\big(\frac{X}{a}\big)^{1/2} \big(\frac{d^2}{a}\big)^{1/2}\big(1+\frac{h}{X}+ \frac{d}{a\sigma}\big)^{1/2}h^{\theta}\bigg(1+ \big(\frac{X}{h}\big)^{\theta}\bigg)$$
 and therefore   \eqref{6.5} is bounded by
    $dX^{1/2+\theta+\epsilon}/a.$
     This finishes the proof of \eqref{fin} and using \eqref{finn} finishes the proof of the theorem.
\end{enumerate}
\end{proof}
\subsection*{Acknowledgements}
I would like to thank my supervisor Nathan Ng for introducing the problem and all the support during my work. Also, I am thankful to Amir Akbary and Gergely Harcos for their helpful comments. 

{\it E-mail address}: farzad.aryan@uleth.ca
\end{document}